\documentclass[dvips]{imsart}

\usepackage{amsthm,amsmath}
\usepackage{amsmath,amsthm,amsfonts,epsfig,amssymb,eucal,eufrak}
\usepackage{enumitem}
\usepackage{color}
\usepackage{bbm}

\newtheorem{lem}{Lemma}
\newtheorem{thm}{Theorem}

\doi{10.1214/154957804100000000}
\pubyear{0000}
\volume{0}
\firstpage{0}
\lastpage{0}
\arxiv{1302.0924v3}

\startlocaldefs
\endlocaldefs

\begin{document}

\begin{frontmatter}

\title{On the Nile Problem by Sir Ronald Fisher
}
\runtitle{Fisher Nile Problem}


\begin{aug}
\author{\fnms{Abram M.} \snm{Kagan}\ead[label=e1]{amk@math.umd.edu}}
\and
\author{\fnms{Yaakov} \snm{Malinovsky}\thanksref{t2}\ead[label=e2]{yaakovm@umbc.edu}}

\address{Department of Mathematics, University of Maryland, College Park, MD 20742, USA\\
and\\
Department of Mathematics and Statistics, University of Maryland, Baltimore County, Baltimore, MD 21250, USA\\
\printead{e1,e2}}

\thankstext{t2}{Corresponding Author}
\runauthor{Kagan and Malinovsky}


\end{aug}


\begin{abstract}
The Nile problem by Ronald Fisher may be interpreted as the problem of making statistical inference
for a special curved exponential family when the minimal sufficient statistic is incomplete. The problem itself and its versions for general curved exponential families
pose a mathematical-statistical challenge: studying  the subalgebras of ancillary statistics within the $\sigma$-algebra of the (incomplete) minimal sufficient statistics
and closely related questions of the structure of UMVUEs.

In this paper a new method is developed that, in particular, proves that in the classical Nile problem no statistic subject to mild natural conditions is a UMVUE.
The method almost solves an old problem of the existence of UMVUEs. The method is purely statistical (vs. analytical) and works for any family possessing an ancillary statistic.
It complements an analytical method that uses only the first order ancillarity (and thus works when the existence of ancillary subalgebras is an open
problem) and works for curved exponential families with polynomial constraints on the canonical parameters of which the Nile problem is a special case.
\end{abstract}



\received{\smonth{6} \syear{2013}}
\tableofcontents
\end{frontmatter}

\section{Introduction}
\label{sec:intr}
The so called Nile problem formulated by Fisher gave rise to interesting mathematical -statistical problems.
The original statement of the problem in Fisher's unique style is in Fisher \cite{r11} (it is cited verbatim in Fisher \cite{r12}, pp. 122):\\
{\it
The agricultural land of a pre-dynastic Egyptian village is of
unequal fertility. Given the height to which the Nile will rise, the
fertility of every portion of it is known with exactitude, but the
height of the flood affects different parts of the territory unequally.
It is required to divide the area, between the several households of
the village, so that the yields of the lots assigned to each shall be in
pre-determined proportion, whatever may be the height to which
proportion the river rises.
}

Fisher himself (\cite{r12}, pp. 169) specified the problem as making statistical inference for a population with density
\begin{equation}
\label{eq:Fisher}
\displaystyle
f(x,y;\theta)=e^{-\left(x \theta+y/\theta\right)},\,\,\, x>0, y>0,
\end{equation}
with $\displaystyle \theta>0$ as a parameter.

If $\displaystyle \left(\left(X_1, Y_1\right),\ldots,\left(X_n, Y_n\right)\right)$ is a sample from population \eqref{eq:Fisher},
the pair $\displaystyle \left(\overline{X}, \overline{Y}\right)$ of the sample means is an incomplete (minimal) sufficient statistic
for $\displaystyle \theta$. Due to incompleteness, there might exist (and in this setting actually exists) an ancillary statistic, (i.e., a statistic
whose distribution does not depend on the parameter),
in the $\displaystyle \sigma$-$\displaystyle \text{algebra}$ $\displaystyle \sigma\left(\overline{X}, \overline{Y}\right)$ generated by $\displaystyle \left(\overline{X}, \overline{Y}\right)$.

Due to incompleteness of the minimal sufficient statistic, the existence and construction of UMVUEs do not follow from the Rao-Blackwell and Lehmann-Scheff\'{e} theorems and become a nontrivial problem which requires a new approach. Nayak and Sinha \cite{r35} mentioned the existence of UMVUEs in the models \eqref{eq:Fisher} and \eqref{eq:NS} (see below)
as an open problem.

Another interpretation of the Nile problem due to Flatto and Shepp \cite{r13} is statistical inference on the correlation coefficient $\displaystyle \rho$
of a bivariate Gaussian vector with density
\begin{equation}
\label{eq:BGauss}
{\displaystyle
\varphi\left(x,y;\rho\right)=\frac{1}{2\pi \sqrt{1-\rho^2}}e^{-\frac{x^2+y^2-2\rho xy}{2(1-\rho^2)}}.
}
\end{equation}
The minimal sufficient statistic for $\displaystyle \rho$ is $\displaystyle \left(X^2+Y^2, XY\right)$.
Note that the pair $\displaystyle \left(X^2+Y^2, XY\right)$ is in one-to-one correspondence with a quadrable $\displaystyle \left( (X,Y), (Y,X),\right.$
\\$\left. (-X,-Y), (-Y,-X)\right)$.
If $\displaystyle \left(X_1, Y_1\right),\ldots,\left(X_n, Y_n\right)$ is a sample from \eqref{eq:BGauss},
the minimal sufficient statistic for $\displaystyle \rho$ is $\displaystyle \left(\sum_{i=1}^{n}(X_i^2+Y_i^2), \sum_{i=1}^{n}X_iY_i\right)$.
The minimal sufficient statistic is again incomplete (even in case of $n=1$), and the problem of
the existence of an ancillary statistic in the $\sigma$-algebra $\displaystyle \sigma\left(\sum_{i=1}^{n}(X_i^2+Y_i^2), \sum_{i=1}^{n}X_iY_i\right)$ generated by
$\displaystyle \left(\sum_{i=1}^{n}(X_i^2+Y_i^2), \sum_{i=1}^{n}X_iY_i\right)$ remains open (see Lehmann and Romano \cite{r33}, pp. 397-398).

The following observation by Flatto and Shepp \cite{r13} solves a related, but different problem.
Let $\displaystyle A$ be a set in $\displaystyle \mathbb{R}^2$ with finite Lebesgue measure, $\displaystyle \lambda(A)<\infty$.
If $\displaystyle A$ is ancillary, i.e.,
\begin{equation*}
\displaystyle \iint_{A} \varphi\left(x,y;\rho\right)dxdy=c,\,\,\,\text{a constant},
\end{equation*}
then $\displaystyle \lambda(A)=0$ (and thus $c=0$). The condition $\displaystyle \lambda(A)<\infty$ which is essential in the proof, seems
artificial from the statistical point of view.
However, a first order ancillary statistic $\displaystyle H(X,Y)$ (i.e., such that $\displaystyle E_{\rho}H(X,Y)=\text{const}$) measurable with respect to $\displaystyle \sigma\left({X^2+Y^2}, {XY}\right)$ exists.
Indeed, set
$$\displaystyle H(x,y)=H(x)+H(y),$$
where $\displaystyle H(u)=\mathbbm{1}_{\{|u|\leq 1\}}.$
One can easily see that $\displaystyle H(x,y)=H(y,x)=H(-x,-y)$ and $\displaystyle E_{\rho}H\left(X,Y\right)=\text{const}$
since the marginal distributions of $\displaystyle X$ and $\displaystyle Y$ do not involve $\displaystyle \rho$.
Though $\displaystyle \int\int H(x,y)dxdy=\infty$, the finiteness of this integral is not required in the definition of the first order ancillarity.

From general results on curved exponential family in Kagan and Palamodov \cite{r20, r21}
(see also supplement in Linnik \cite{r34}, and for another proof see Unni \cite{r39})
it follows that the only UMVUEs from a sample $\displaystyle \left(\left(X_1, Y_1\right),\ldots,\left(X_n, Y_n\right)\right)$
from \eqref{eq:BGauss} are constants.

Note that from the analytical point of view, inference problems for a sample $\displaystyle (X_{1},\ldots,X_{n})$ from a population with density
\begin{equation}
\label{eq:NS}
\displaystyle
f(x;\theta)=\frac{1}{\sqrt{2\pi}c\theta}e^{-\frac{(x-\theta)^2}{2c^2\theta^2}}
\end{equation}
with $\displaystyle \theta>0$ as a parameter, $\displaystyle c>0$ known, are very close to the Nile problem.
The assumption that the standard deviation is proportional to the mean seems reasonable in the setup of direct measurements.
These problems were studied in a number of papers
(see, e.g., Khan \cite{r30}, Gleser and Healy \cite{r15}, Hinkley \cite{r17}).
The minimal sufficient statistic for $\displaystyle \theta$ is a pair $\displaystyle \left(\overline{X}, S\right)$
of the sample mean and standard deviation. The sufficient statistic is incomplete and there exists a convenient ancillary statistic in $\displaystyle \sigma\left(\overline{X}, S\right)$.
Combining this with results from Rao \cite{r38} on the structure of UMVUEs and Kagan \cite{r19}, Barra \cite{r3} and Bondesson \cite{r6} on sufficiency, we prove by purely statistical
tools that the only UMVUEs are constants. The statistical method works for any family $\mathcal{P}=\{P_{\theta}, \theta \in \Theta\}$  of distributions on $(\mathcal{X}, \mathcal{A})$ possessing an ancillary subalgebra $\mathcal{B}$, i.e., such that $P_{\theta}(B)=\text{const}$ in $\theta$ for all $B \in \mathcal{B}$.

An analytical proof using general results for curved exponential families
can be found in Kagan and Palamodov \cite{r20, r21, r22} and the dissertation of Unni \cite{r39}. In Section \ref{se:ex} we discuss some examples that cannot be treated by the analytical method but yield to the statistical method.

The following observation is due to I. Pinelis (private communication). Let $\displaystyle X$ be distributed according to \eqref{eq:NS}. Then the statistic
$\displaystyle \mathbbm{1}_{\{X>0\}}$ is ancillary, so that $\displaystyle X$ is incomplete. He conjectured that if the parameter space is $\displaystyle \mathbb{R}\setminus\{0\}$,
then $X$ is complete. If so, it is an interesting phenomenon.

Extrapolating from the above three different interpretations of the Nile problem, the following problem seems to be of a general interest. Let
$\displaystyle \mathcal{P}=\{P_{\theta},\,\,\theta \in \Theta\}$ be a family of probability distributions on a measurable space
$(\mathcal{X}, \mathcal{A})$, and let $T:(\mathcal{X}, \mathcal{A})\rightarrow (\mathcal{T}, \mathcal{C})$ be an incomplete sufficient statistic for $\theta$.
Set $\mathcal{\widetilde{A}}=T^{-1}\mathcal{C}$. Describe, if they exist, $\mathcal{\widetilde{A}}$-measurable (i.e., function of $T$) ancillary statistics.

\section{Sufficiency, Ancillarity and UMVUEs}
Let $\mathcal{P}=\{P_{\theta}, \theta \in \Theta\}$ be a family of probability distributions parameterized by a general parameter $\theta$
of a random element $X$ taking values in a measurable space $(\mathcal{X}, \mathcal{A})$.
A subalgebra $\mathcal{B} \subset \mathcal{A}$ is called ancillary if $P_{\theta}(B)=\text{const}$ in $\theta$ for all $B \in \mathcal{B}$.
A statistic $T(X)$ is called ancillary if the the subalgebra it generates is ancillary.
A statistic $T(X)$ taking values in $\mathbb{R}$ with $E_{\theta}|T(X)|<\infty$ is called a first order ancillary
if $E_{\theta}T(X)=\text{const}$ in $\theta$. A well known theorem due to Basu \cite{r4, r5} says that if $\mathcal{P}$
is a linked family, i.e., for any pair $\theta^{'}, \theta^{''}\in \Theta$ there exist a sequence $\theta_{0}, \theta_{1},\ldots,\theta_{n}, \theta_{n+1}$
with $\theta_{0}=\theta^{'}, \theta_{n+1}=\theta^{''}$ such that $P_{\theta_{j}}, P_{\theta_{j+1}}$ are not mutually singular
(i.e., $P_{\theta_{j}}(A)=1 \Rightarrow P_{\theta_{j+1}}(A)>0$), then any subalgebra $\mathcal{B}$ which is $\mathcal{P}$-independent
of a sufficient subalgebra $\mathcal{\widetilde{A}}$ (i.e., $P({\widetilde{A}}\cap B)=P({\widetilde{A}})P(B)$ for any
$\widetilde{A}\in \mathcal{\widetilde{A}}, B\in \mathcal{B}$ and $\theta \in \Theta$ ) is ancillary.
A straightforward conversion of the Basu result is plainly false: there exist ancillary subalgebras within the algebra of sufficient statistics (see examples 1, 2 below).

Note in passing that if $\mathcal{C}$ is a complete sufficient algebra, then any ancillary algebra is $\mathcal{P}$-independent
of $\mathcal{C}$ (this result does not require $\mathcal{P}$ to be a linked family). In this case there is no $\mathcal{C}$-measurable
ancillary statistic.

\noindent
{\bf Example 1}.\\
Let $\displaystyle \left(\left(X_1, Y_1\right),\ldots,\left(X_n, Y_n\right)\right)$ be a sample from
$$f(x,y;\theta)=e^{-\left(x \theta+y/\theta\right)},\,\,\, x>0, y>0, \theta>0.$$
The minimal sufficient statistic is ($\overline{X},\overline{Y}$), and one can easily see that $\overline{X}\,\overline{Y}$
is an ancillary statistic.\\
\noindent
{\bf Example 2}.\\
Let $\displaystyle (X_{1},\ldots,X_{n})$ be a sample from
$$f(x;\theta)=\frac{1}{\sqrt{2\pi}\theta}e^{-\frac{(x-\theta)^2}{2\theta^2}},\,\,\, \theta>0.$$
The minimal sufficient statistic is ($\overline{X}, S$), and one can check that the statistic $\overline{X}/S$
is ancillary.

However, if a subalgebra $\mathcal{C}\subset \mathcal{A}$ which is $\mathcal{P}$-independent of an ancillary subalgebra $\mathcal{B}$
is large enough, then $\mathcal{C}$ is sufficient for $\mathcal{P}$. This observation is due to Kagan \cite{r19}
and independently Barra \cite{r3} and Bondesson \cite{r6}.
\begin{lem}
\label{lem:KB}
Suppose that a subalgebra $\mathcal{C}$ is $\mathcal{P}$-independent of an ancillary algebra $\mathcal{B}$
and together with $\mathcal{B}$ generates $\mathcal{A}$, i.\,e., $\mathcal{A}$ is the smallest $\sigma$-algebra that contains both
$\mathcal{C}$ and $\mathcal{B}$, $\displaystyle \sigma(\mathcal{C}, \mathcal{B})=\mathcal{A}$. Then $\mathcal{C}$ is sufficient for $\mathcal{P}$.
\end{lem}
For the sake of completeness, a short proof of Lemma \ref{lem:KB}  is given in Appendix (the result was proved in Kagan \cite{r19}, Barra \cite{r3} and Bondesson \cite{r6}).

Based on a paragraph in Fisher \cite{r12}, pp. 168, it is likely that Fisher's definition of an ancillary algebra $\mathcal{B}$ required existence of
a $\mathcal{P}$-independent complement $\mathcal{C}$, i.e., that $\displaystyle \sigma(\mathcal{C}, \mathcal{B})=\mathcal{A}$.
If so, he knew that $\mathcal{C}$ is sufficient for $\mathcal{P}$.

Basu's \cite{r4, r5} theorem was useful in characterization of distributions by independence of statistics
(see, e.g., Ferguson \cite{r9, r10}, Klebanov \cite{r29}, Kagan \cite{r24} and an expository paper by Gather \cite{r14}).
Here we want to demonstrate that combining Lemma \ref{lem:KB} with Rao's result \cite{r38} on the structure of UMVUEs proves triviality of
UMVUEs in the models \eqref{eq:Fisher} and \eqref{eq:NS}.
The proof which is purely statistical seems new and is of interest in its own. An analytical method covering curved exponential families with polynomial constraints on the natural parameters was developed in Kagan and Palamodov \cite{r20, r21} and simplified in Unni \cite{r39}.
It is based on a result by Wijsman \cite{r40} on the existence of the first order ancillary statistics for samples from curved exponential
families with polynomial constraints on the natural parameters.

To state Rao's result, recall that if an observation $X\sim P_{\theta}$ with $\theta \in \Theta$ as a parameter, a statistic $g(X)$ with
$\displaystyle E_{\theta}|g(X)|^2<\infty$ is UMVUE iff it is uncorrelated with any unbiased estimator of zero $U(X)$ with $\displaystyle E_{\theta}|U(X)|^2<\infty$,
i.\,e., $\displaystyle E_{\theta}\left\{g(X)U(X)\right\}=0,\,\,\, \theta \in \Theta$
(unbiased estimators of zero are also called zero-mean statistics).
Rao \cite{r38} observed that if a statistic $\displaystyle T=g(X)$ is a UMVUE, then, provided that $\displaystyle E_{\theta}|g(X)U(X)|^2<\infty$,
$T^2$ is also a UMVUE. Proceeding in the same way under the assumption $\displaystyle E_{\theta}|g(X)|^{k}<\infty,\,\,\,\,k=1,2,\ldots$
we observe that any polynomial of $T$ is UMVUE. Assuming moreover that the polynomials of $T$ are complete in $\displaystyle L_{\theta}^2\left(T(X)\right)$,
the Hilbert space of functions $h(T)$ with $\displaystyle E_{\theta}|h(T)|^{2}<\infty$, one gets that any statistic $\displaystyle h(T)$ with $\displaystyle E_{\theta}|h(T)|^{2}<\infty$ is a UMVUE (actually, the UMVUE of $\displaystyle E_{\theta}h(T)$).

In particular, if $S$ is an ancillary statistic, the $\sigma$-algebras $\sigma({T})$ and $\sigma({S})$ are independent for all $\displaystyle \theta \in \Theta$.
If the pair $\displaystyle (T, S)$ determines the sample point $X$ or, equivalently, $\sigma({T},{S})=\sigma(X)=\mathcal{A}$, then $T$ is sufficient for $\theta$ by virtue of Lemma~\ref{lem:KB}.

In the known examples (Lehmann and Scheff\'{e} \cite{r31}, reproduced in Lehmann and Casella \cite{r32}, p. 84, Bondesson \cite{r7}, Kagan and Konikov \cite{r26}), the UMVUEs form a subalgebra $\mathcal{E}\subset \mathcal{A}$ (actually, a subalgebra of the minimal sufficient subalgebra $\mathcal{\widetilde{A}}$) called the $\sigma$-algebra of UMVUEs, i.e., all $\mathcal{E}$-measurable statistics with finite second moments and only they are UMVUEs.

As these examples demonstrate, the problem of describing the algebra of UMVUEs for a general family of distributions seems rather difficult.
In these examples, the minimal sufficient statistic is trivial (i.e., coincides with $\mathcal{A}$), while $\mathcal{E}$ is not.

\section{Statistical Method and the Original Nile Problem}
We shall start with a general result on UMVUEs.
Let $\mathcal{P}=\{P_{\theta}, \theta \in \Theta\}$ be a family of distributions on $(\mathcal{X}, \mathcal{A})$, $\mathcal{C} \subset \mathcal{A}$
an (incomplete) sufficient subalgebra, and $\mathcal{B} \subset \mathcal{C}$ an ancillary subalgebra. In terms of statistics, $X$ represents the data, $T=T(X)$ is an incomplete (minimal) sufficient statistic generating $\mathcal{C}$, and $W=W(T)$ is an ancillary statistic.

We present a new method for obtaining a strong necessary condition for a statistic to be UMVUE.

Here are the conditions imposed on a statistic $g=g(T)$.\\
Condition 1.\\
\begin{equation}
E_{\theta}|g(T)|^{k}<\infty, k=1,2,\ldots,\,\,\,\theta \in \Theta.
\end{equation}
Condition 2.\\
\begin{equation}
\text{span}_{\theta}\{1, g, g^2,\ldots\}=L^{2}_{\theta}(g),
\end{equation}
where $\displaystyle L^{2}_{\theta}(g)$ denotes the Hilbert space of function $h\left(g(T)\right)$ with the inner product
$\displaystyle \left(h_1, h_2\right)_{\theta}=E_{\theta}\left(h_{1}h_{2}\right)$ and $\displaystyle \text{span}_{\theta}\{1, g, g^2,\ldots\}$ is the closure
in $\displaystyle L^{2}_{\theta}(g)$ of the sums $\displaystyle \sum_{j}c_{j}g^{j}$.\\
Condition 3a.\\
\begin{equation}
\sigma\left(g(T)\right)\neq\sigma\left(T\right)\,\,\,\text{(strict\,\,\,inclusion)}.
\end{equation}
Condition 3b.\\
\begin{equation}
\sigma\left(g(T),W\right)=\sigma\left(T\right).
\end{equation}
Conditions 3a and 3b refer to the $\sigma$-algebras, one generated by $\displaystyle g\left(T\right)$
and the other by ancillary statistic $\displaystyle W$.
Roughly speaking, Condition 3a means that $g(T)$ is not a one-to-one function, while the pair $\left(g\left(T\right), W(T)\right)$ is.

In real setups, an incomplete $T$ is multidimensional, while $g(T)$ is a scalar valued statistic.

Let $\displaystyle Q_{\theta}(u)=P_{\theta}\left\{g\left(T\right)\leq u\right\}$ be the distribution of $\displaystyle g\left(T\right)$. If for all $\theta \in \Theta$, $\displaystyle Q_{\theta}(u)$ is continuous in $u$ and the moment problem for $\displaystyle Q_{\theta}$ is determinate, i.e.,
$\displaystyle Q_{\theta}$ is the only distribution with the moments
$$\displaystyle \alpha_{m}(\theta)=\int u^{m} dQ_{\theta}(u)=E_{\theta}\left\{g^{m}\left(T\right)\right\},\,\,\,m=0,1,2,\ldots$$
then the polynomials in $u$ are dense in $\displaystyle L^{2}(Q_{\theta})$ or, equivalently, polynomials in $g(T)$ are dense in $\displaystyle L^{2}_{\theta}(g)$.
A sufficient condition for this is given by the Carleman's classical criterion: $\sum_{m=1}^{\infty}\left(\alpha_{2m}(\theta)\right)^{-1/(2m)}=\infty.$
See Corollary 2.3.3 in Akhiezer \cite{r1}, p. 45.

\begin{thm}
\label{th:1}
A statistic $\displaystyle g\left(T\right)$ satisfying Conditions 1,\,2,\,3a,\,3b cannot be a UMVUE.
\end{thm}
\begin{proof}
Suppose that $\displaystyle g=g\left(T\right)$ is a UMVUE.
Then for any zero-mean statistic $\displaystyle U=U\left(T\right)$
with $\displaystyle E_{\theta}|U|^{2}<\infty$ one has
\begin{equation}
\label{eq:U}
E_{\theta}\left(gU\right)=0, \,\,\,\,\theta \in \Theta.
\end{equation}
In particular, since $\displaystyle W$ is an ancillary statistic, \eqref{eq:U} implies
\begin{equation}
\label{eq:W}
E_{\theta}\left(g\,|\,W\right)=E_{\theta}(g), \,\,\,\,\theta \in \Theta.
\end{equation}
Turn now to $g^2$. For any bounded non-zero statistic $U,$ the statistic $\displaystyle U_{1}=gU$ is, by virtue of \eqref{eq:U} and Condition 1
also zero-mean statistic with finite second moment, $\displaystyle E_{\theta}(U_{1})=0, \,\,\,E_{\theta}(|U_{1}|^2)<\infty,\,\,\,\theta \in \Theta.$
Thus, from $\displaystyle g=g\left(T\right)$ being a UMVUE, follows that
$\displaystyle E_{\theta}\left(gU_{1}\right)=E_{\theta}\left(g^2U\right)=0,\,\,\theta \in \Theta,$ implying
\begin{equation}
\label{eq:WW}
E_{\theta}\left(g^2\,|\,W\right)=E_{\theta}(g^2), \,\,\,\,\theta \in \Theta.
\end{equation}
Proceeding in the same way, one can prove that
\begin{equation}
\label{eq:WW}
E_{\theta}\left(g^k\,|\,W\right)=E_{\theta}(g^k),\,\,\,k=1,2,3\ldots; \,\,\,\,\theta \in \Theta.
\end{equation}
Notice again that the above arguments are essentially due to Rao \cite{r38}.
Let now $\displaystyle h=h\left(T\right)\in L^{2}_{\theta}(g)$. Take a sequence of the finite sums
$\displaystyle \sum_{k=1}^{m}c_{k,m}\,g^{k}$ with
\begin{equation*}
E_{\theta}\left(\sum_{k=1}^{m}c_{k,m}\,g^{k}-h\right)^2\rightarrow 0,\,\,\,\text{as}\,\,\,\,m\rightarrow \infty.
\end{equation*}
Such sequence exists due to Condition 2.
Since $$\displaystyle \left|E_{\theta}\left(\sum_{k=1}^{m}c_{k,m}\,g^{k}-h\right)\right|^2\leq E_{\theta}\left(\left|\sum_{k=1}^{m}c_{k,m}\,g^{k}-h\right|^2\right)$$
one has
$
E_{\theta}\left(\sum_{k=1}^{m}c_{k,m}\,g^{k}\right)\rightarrow E_{\theta}\left(h\right),\,\,\,\text{as}\,\,\,\,m\rightarrow \infty.
$
Furthermore,
\begin{align*}
&
E_{\theta}\left\{E_{\theta}\left(\sum_{k=1}^{m}c_{k,m}\,g^{k}-h\,|\,W\right)\right\}^2=
E_{\theta}\left\{E_{\theta}(\sum_{k=1}^{m}c_{k,m}\,g^{k}\,|\,W)-E_{\theta}(h\,|\,W)\right\}^2\\
&
=
E_{\theta}\left\{E_{\theta}(\sum_{k=1}^{m}c_{k,m}\,g^{k})-E_{\theta}(h\,|\,W)\right\}^2
\leq
E_{\theta}\left(\sum_{k=1}^{m}c_{k,m}\,g^{k}-h\right)^2\rightarrow 0,
\end{align*}
as $m\rightarrow \infty$.
Hence, $\displaystyle E_{\theta}\left(h\,|\,W\right)=E_{\theta}\left(h\right).$
Therefore, any $\displaystyle h\in L^{2}_{\theta}(g)$ has a constant conditional expectation on $\displaystyle \sigma(W)$, i. e.,
$\sigma\left(g\left(T\right)\right)$ and $\displaystyle \sigma(W)$ are independent for any $\theta \in \Theta$.

By virtue of Lemma \ref{lem:KB} and Condition 3b, the statistic $\displaystyle g\left(T\right)$ (or, equivalently,
$\sigma$-algebra $\sigma\left(g\left(T\right)\right)$) is sufficient for $\theta$. Actually, $\sigma\left(g\left(T\right)\right)$ is complete sufficient for $\theta$. Indeed, if $E_{\theta}\left\{U\left(g(T)\right)\right\}=0$ identically in $\theta$, then $U$ is an unbiased estimator of zero. As a function of $g(T)$, $U$ is a UMVUE. Plainly, the UMVUE of zero is zero, so that $P_{\theta}\left(U=0\right)=1$ for all $\theta \in \Theta$. But due to Condition 3a, $\sigma\left(g\left(T\right)\right)$ is a proper subalgebra of the minimal sufficient $\sigma$-algebra $\sigma\left(T\right)$, which is a contradiction.
Notice in conclusion that the trivial UMVUE's, $g\left(T\right)=\text{const}$, do not satisfy Condition 3b.
\end{proof}
Let now $\displaystyle \left(X_1, Y_1\right),\ldots,\left(X_n, Y_n\right)$ be a sample from
$$f(x,y;\theta)=e^{-\left(x \theta+y/\theta\right)},\,\,\, x>0, y>0,$$
with $\displaystyle \theta>0$ as a parameter. The inference from the above sample is what is usually referred to as the Nile problem by Ronald Fisher.
Plainly, the vector $\displaystyle T=(\overline{X}, \overline{Y})$ is the minimal sufficient statistic for $\displaystyle \theta$ and $\displaystyle W=\overline{X}\,\overline{Y}$
is an ancillary statistic. The minimal sufficient statistics is incomplete and this makes the problem of existence and description of UMVUEs nontrivial.

Applied to the Nile problem, Theorem \ref{th:1} proves
that a statistic satisfying rather general ``regularity type" conditions (Conditions 1,\,2,\,3a,\,b) is { not} a UMVUE (the existence of a nonconstant
UMVUE is an open problem, according to Nayak and Sinha \cite{r35}).

Turn now to a natural class of estimators of $\theta$.
A statistic $\displaystyle \widetilde{\theta}\left(\overline{X}, \overline{Y} \right)$ is called an equivariant estimator of $\theta$ if
\begin{equation}
\label{eq:Inv}
\widetilde{\theta}\left(\overline{X}/\lambda, \overline{Y}\lambda \right)=\lambda\widetilde{\theta}\left(\overline{X}, \overline{Y}\right)\,\,\, \text{for any}\,\,\, \lambda>0.
\end{equation}
The equivariant estimators in the Nile problem were studied in Kariya \cite{r28}.
Plainly, an equivariant estimator can be written as
\begin{equation}
\label{eq:Equiv}
\widetilde{\theta}=\overline{Y}h(W)
\end{equation}
for some $h$. Here $\overline{Y}$ here may be replaced with any statistic of degree of homogeneity one in sense of \eqref{eq:Inv},
e. g., $\displaystyle \sqrt{\overline{Y}/\overline{X}}$ (the latter is the maximum likelihood estimator (MLE) of $\theta$, as noticed by Fisher himself), or $\displaystyle {1/\overline{X}}$ .
If \eqref{eq:Equiv} is a UMVUE, then \eqref{eq:WW} results in
\begin{equation}
\label{eq:h}
h(W)=\frac{1}{E_{1}\left(\overline{Y}\,|\,W\right)},
\end{equation}
where $E_{1}$ is the expectation taken when $\theta=1$ and
\begin{equation}
\label{eq:form}
E_{1}\left(\overline{Y}\,|\,W=w\right)=\frac{\int_{0}^{\infty}\frac{1}{z^2}e^{-n\left(\frac{1}{z}+wz\right)}dz}
{\int_{0}^{\infty}\frac{1}{z}e^{-n\left(\frac{1}{z}+wz\right)}dz}.
\end{equation}
If $\displaystyle \widetilde{\theta}$ is a UMVUE, so is $\displaystyle \widetilde{\theta}^{\,\,2}$ and thus
$\displaystyle E\left({\widetilde{\theta}}^{\,\,2}\,|\,W\right)=E\left({\widetilde{\theta}}^{\,\,2}\right)$, but\\
$\displaystyle E\left({\widetilde{\theta}}^{\,\,2}\,|\,W\right)=h^2(W)E\left(\overline{Y}^{\,2}\,|\,W\right)=\theta^2
\frac{{E_{1}\left(\overline{Y}^{\,2}\,|\,W\right)}}{\left({E_{1}\left(\overline{Y}\,|\,W\right)}\right)^2}$ and straightforward calculations using
\eqref{eq:form} show that $\displaystyle E\left({\widetilde{\theta}}^{\,\,2}\,|\,W\right)$ depends on $W$ . But this is in contradiction with necessary condition \eqref{eq:WW} for being a UMVUE. Thus, \eqref{eq:Equiv}
is not a UMVUE.
Similar arguments show that no estimator of the form $\displaystyle \overline{Y}^{\,k}h(W)$ or $\displaystyle \overline{X}^{\,k}h(W)$
is a UMVUE.

\section{Problems Closely Related to the Nile Problem}
The following setup of direct measurements, being of an interest in its own, has the same basic features as the Nile problem:
an incomplete minimal sufficient statistic and an ancillary statistic which is a function of the sufficient one.

Let $\displaystyle \left(X_1,X_2,\ldots,X_n\right)$ be a sample from a normal population $\displaystyle N(\theta, c^2\theta^2)$ with $\theta,\, \theta>0$ as a parameter, and $c>0$ is known.
In other words,
$$X_i=\theta+\varepsilon_i,\,\,\,i=1,\ldots,n,$$
where $\displaystyle \varepsilon_i,\ldots,\varepsilon_n$ are independent random variables distributed as $\displaystyle N(0, c^2\theta^2)$.
In the standard setups of direct measurements, the distribution function $F(x)$ of $\varepsilon_i$ (not necessarily normal) is assumed independent of $\theta$
so that $\displaystyle \left(X_1,X_2,\ldots,X_n\right)$ is a sample from $F(x-\theta)$ with a location parameter $\theta$. Estimation
of a location parameter in small samples was originated in Pitman \cite{r36}. Since then it has been thoroughly studied, especially for the quadratic loss function;
see, e. g., monographs Lehmann and Casella \cite{r32}, Casella and Berger \cite{r8}, Kagan \cite{r23}, Zacks \cite{r41} and recent papers Kagan and Rao \cite{r25}, Kagan {\it et al} \cite{r27} and references therein.
A special role of normal distribution $\displaystyle \Phi(x)$ in estimation of a location parameter is due to the fact that in the class of the distributions $F$ with a given
variance $\sigma^2$, the Fisher information on $\theta$ contained in an observation $X_i \sim F(x-\theta)$ is minimized at $F(x)=\Phi(x/\sigma)$, with
the minimum equals $\displaystyle 1/\sigma^2$. A closely related result is that under the quadratic loss function and $n\geq3$, $\displaystyle \overline{X}$ is an admissible estimator of
$\theta$ if and only if $F=\Phi$. The if part is due to Hodges and Lehmann \cite{r16} and only if part due to Kagan {\it et al}. \cite{r18}.

The setup with $\displaystyle \left(X_1,X_2,\ldots,X_n\right)$ being a sample from $\displaystyle N(\theta, c^2\theta^2)$ differs significantly from
that with $\displaystyle \left(X_1,X_2,\ldots,X_n\right)$ taken from $\displaystyle N(\theta, \sigma^2)$ with $\sigma^2$ independent of $\theta$.
Firstly, the Fisher information on $\theta$ in $X_i\sim N(\theta, c^2\theta^2)$ equals $\displaystyle \frac{1}{\theta^2}\left(2+\frac{1}{c^2}\right)$,
and it exceeds the information in $\displaystyle X_i \sim N(\theta, \sigma^2)$ calculated at $\sigma^2=c^2\theta^2$ which equals $\displaystyle \frac{1}{c^2\theta^2}$.
Secondly, the minimal sufficient statistic for a sample from $\displaystyle N(\theta, c^2\theta^2)$, is the pair $\displaystyle (\overline{X}, S^2)$ of the sample mean and variance, and it is incomplete, while from a sample from $\displaystyle N(\theta, \sigma^2)$ with known $\sigma^2$ it is  $\displaystyle \overline{X}$ and it is complete
and for sample from $\displaystyle  N(\theta, \sigma^2)$ with $\displaystyle (\theta, \sigma^2)\in (\mathbb{R}\times\mathbb{R}_{+} )$, the pair $\displaystyle (\overline{X}, S^2)$ is a complete sufficient statistic.

The setup of small and large samples from $\displaystyle N(\theta, c^2\theta^2)$ was studied in a number of papers.
Khan \cite{r30} found the best unbiased estimator of $\theta$ in the class of estimators linear in $\displaystyle \overline{X}, S$  and showed that it is asymptotically efficient.
Gleser and Healy \cite{r15} proved admissibility of the best (scale)-equivariant estimator of $\theta$. Since it is different from the (also equivariant)
MLE, the latter is inadmissible. See also Hinkley \cite{r17} and Kariya \cite{r28} for related results.
According to Nayak and Sinha \cite{r35}, the problem of existence of UMVUEs is open.

Since in samples from a normal population $\displaystyle \overline{X}$ and $\displaystyle S$ are independent, the setup of sampling from
$\displaystyle N(\theta, c^2\theta^2)$ is very similar to the Nile problem:  $T=\displaystyle (\overline{X}, S^2)$ is an incomplete sufficient statistic and
$\displaystyle W=\frac{\overline{X}}{S}$ is an ancillary statistic. To show the latter, write
\begin{align*}
&\frac{\overline{X}}{S}=\frac{(\overline{X}-\theta+\theta)/\theta}{S/\theta}=\frac{\frac{\overline{X}-\theta}{\theta}+1}{S/\theta}
\end{align*}
and notice that the distributions of $\displaystyle \frac{\overline{X}-\theta}{\theta}$ and $\displaystyle \frac{S}{\theta}$ do not depend on $\displaystyle \theta$.
A direct application of the method used in proving Theorem \ref{th:1} proves the following result.
\begin{thm}
\label{th:2}
Let $\displaystyle g\left(T\right)$ be a statistic satisfying Conditions 1,\,2,\,3a,\,3b with $T=\displaystyle (\overline{X}, S)$ is an incomplete sufficient statistic and
$\displaystyle W=\frac{\overline{X}}{S}$ is an ancillary statistic. Then $\displaystyle g\left(T\right)$ is not a UMVUE.
\end{thm}

\section{Analytical Method}
We shall show now that Theorem \ref{th:2} holds true without (unnecessary) Conditions 1,\,2,\,3a,\,3b due to the fact that the family of normal distributions  $\displaystyle N(\theta, c^2\theta^2)$ with $\displaystyle \theta$ as a parameter is a curved exponential family with polynomial constraints on the natural parameter. It is straightforward corollary of Lemma \ref{lem:W} below whose proof is purely analytical. The idea of the proof goes back to Wijsman \cite{r40}.
As one can easily see, the probability density function of the minimal sufficient statistic (based on the sample of size $n$) $\displaystyle \left(\sum_{i=1}^{n}X_{i}, \sum_{i=1}^{n}X_{i}^{2}\right)$ at the point $\displaystyle (u, v)$ is
$$f(u,v;\,\eta_1, \eta_2)=h(u,v)e^{\eta_1 u+\eta_2 v-\psi(\eta_1,\eta_2)},\,\,\,u\in \mathbb{R},\,v \in \mathbb{R}_{+},$$
where the explicit form of $\displaystyle h(u,v)$ does not matter, but what matters for our purpose is a constraint $\displaystyle \eta_1^{2}+\frac{2}{c^2}\eta_2=0$ on the natural parameters
$\displaystyle \eta_1=\frac{1}{c^2\theta},\,\,\eta_2=-\frac{1}{2c^2\theta^2}$. The structure of UMVUEs for samples from natural exponential families (NEFs) with polynomial constraints on the parameters was studied in Kagan and Palamodov \cite{r20, r21} and Unni \cite{r39} where the following result was proved.
\begin{lem}
\label{lem:W}
Let the distribution of the vector of sufficient statistics
$T=\left(T_1,\ldots,\right.$\\$\left.T_s\right)$ is given by a density
$$f(t_1,\ldots,t_s;\,\theta_1,\ldots,\theta_s)=
h(t_1,\ldots,t_s)e^{\theta_1t_1+\ldots+\theta_s t_s-\psi(\theta_1,\ldots,\theta_s)}$$
with the parameter set being the intersection $\displaystyle \Xi\cap\Pi$
where $\displaystyle \Xi$ is an open set in $\displaystyle \mathbb{R}^{s}$ and $\displaystyle \Pi$ the algebraic manifold defined by polynomial constraints
\begin{equation}
\label{eq:Const}
\Pi_1\left(\theta_1,\ldots,\theta_s\right)=0,\ldots, \Pi_{m}\left(\theta_1,\ldots,\theta_s\right)=0.
\end{equation}
A statistic $\displaystyle Q(T)$ is a UMVUE if and only if there exists a linear reparametrization
\begin{align}
\label{eq:kp}
&\theta_1=a_{11}\theta_1^{'}+\ldots+a_{1s}\theta_s^{'}\nonumber \\
&\ldots\\
&\theta_s=a_{s1}\theta_1^{'}+\ldots+a_{ss}\theta_s^{'}\nonumber
\end{align}
such that the constraints \eqref{eq:Const} involve only $\displaystyle \theta_1^{'},\ldots,\theta_m^{'},\,m\leq s$ and the remaining components
$\displaystyle \theta_{m+1}^{'},\ldots,\theta_s^{'},\,m\leq s$ run an open set in $\mathbb{R}^{s-m}$ in which case $Q$ depends only on
$$\displaystyle T_{m+1}=a_{1,m+1}T_1+\ldots+a_{s,m+1}T_s,\ldots,T_{s}=a_{1,s}T_1+\ldots+a_{s,s}T_s.$$
\end{lem}
\begin{proof}
See Kagan and Palamodov \cite{r20, r21} and Unni \cite{r39}.
\end{proof}
The sufficiency part simply means that under the new parametrization, $\displaystyle \left(T_{m+1},\right.$\\
$\left.\ldots, T_{s}\right)$ is complete sufficient statistic for
$\displaystyle \left(\theta_{m+1}^{'},\ldots, \theta_s^{'}\right)$ for any fixed values of $\displaystyle \theta_{1}^{'},\ldots, \theta_{m}^{'}$.
Bondesson \cite{r7} noticed that if $\mathcal{P}=\{P_{\theta,\,\eta}\}$ is a family of distributions of a random element $X \in (\mathcal{X}, \mathcal{A})$
parameterized by a ``bivariate" parameter $(\theta,\,\eta)$, and a statistic $T=T(X),\, T:(\mathcal{X}, \mathcal{A})\rightarrow (\mathcal{T}, \mathcal{B})$
is complete sufficient for $\theta$ for any fixed value of $\eta$, then $\sigma(T)$ is an algebra of UMVUEs.

Lemma \ref{lem:W} provides an analytical proof of non-existence of nontrivial UMVUEs from the samples from populations \eqref{eq:Fisher}, \eqref{eq:BGauss} and \eqref{eq:NS}.
All three densities are from exponential families with polynomial constraints on the natural parameters,
$\displaystyle \eta_1=-\theta,\,\,\eta_2=-\frac{1}{\theta}$\, with\,  $\eta_1\eta_2-1=0$ in case of \eqref{eq:Fisher}, $\displaystyle \eta_1=-\frac{1}{2(1-\rho^2)},\,\,\eta_2=\frac{\rho}{1-\rho^2}$\, with\, $2\eta_1-\eta_2^2+4\eta_1^2=0$ in case of \eqref{eq:BGauss}, and
$\displaystyle \eta_1=-\frac{1}{2c^2\theta^2},\,\,\eta_2=\frac{1}{c^2\theta}$ \,with\, $\eta_1+c^2/2\eta_2^2=0$ in case \eqref{eq:NS}.
One can see that reparametrization \eqref{eq:kp} does not exist in all these cases.

\section{Some Applications of the Statistical Method}
The powerful analytical method originated in Wijsman \cite{r40} and developed in Kagan and Palamodov \cite{r20, r21} is applicable to exponential families with polynomial constraints on the canonical parameters (sometimes referred to as curved exponential families). In this section examples of non-exponential families are presented where nonexistence of UMVUEs is proved by the statistical method. 
\\
\noindent
{\bf Example 3}.\\
Let $\displaystyle \left(X_1,\ldots,X_n\right)$ be a sample from a uniform distribution $U\left(\theta-1,\right.$
$\left.\theta+1\right)$ (plainly non-exponential) with $\theta\in \mathbb{R}$ as a parameter. The pair $\left(X_{(1)}, X_{(n)}\right)$ of the minimum and maximum of the sample elements is an (incomplete) minimal sufficient statistic for the parameter $\theta$, while the range $W=X_{(n)}-X_{(1)}$ is an ancillary statistic. Were $g\left(X_{(1)}, X_{(n)}\right)$ a UMVUE and $\sigma\left(g, W\right)=\sigma\left(X_{(1)}, X_{(n)}\right)$, $g$ would have been a complete sufficient statistic for $\theta$ contradicting the minimality of the (bivariate) sufficient statistic
$T=\left(X_{(1)}, X_{(n)}\right)$.
In particular, the Pitman estimator of $\theta$ under the quadratic loss is  $\left(X_{(1)}+X_{(n)}\right)/2$ and though the best in the class of equivariant estimators, it is not UMVUE.

In the similar way one can prove the non-existence of UMVUEs from a sample from  $U\left(\theta, \lambda \theta\right)$ with $\lambda>1$ known, $\theta>0$ as a parameter.
\\
\noindent
{\bf Example 4}.\\
Let $\displaystyle \left(X_1,\ldots,X_n\right)$ be a sample from an arbitrary (not necessarily exponential) location parameter family $F(x-\theta)$ with a known (arbitrary) $F(x)$. Let $t_{n}=t_{n}\left(X_1,\ldots,X_n\right)$ be the Pitman estimator of $\theta$ for the quadratic loss. The statistic $t_{n}$ and the residuals $\left(X_1-\overline{X}_{n},\ldots,X_{n-1}-\overline{X}_{n}\right)$ together determine the sample point. Thus, if $t_n$ is a UMVUE, it is a function of complete sufficient statistic. Hence, for samples from location parameter families we have a complete description of the UMVUEs: they are statistics depending on the data through the complete sufficient statistic.

An alternative `` all or nothing" holds for the location parameter families: either all estimative function of parameter possess UMVUEs or none (except constants).
\\
\noindent
{\bf Example 5}.\\
Let $\displaystyle \left(X_1,\ldots,X_n\right)$ be a sample from an arbitrary location-scale parameters family $F\left((x-\theta)/\sigma\right)$ with $\theta \in \mathbb{R}$ and $\sigma \in \mathbb{R}_{+}$ as parameters. Suppose that a vector $\displaystyle \left(U_1, U_2\right)=\left(U_1(X_1,\ldots,X_n),\right.$
$\left.U_2(X_1,\ldots,X_n)\right)$ is a UMVUE meaning that the covariance matrix $\displaystyle V_{\theta, \sigma }\left(U_1, U_2\right)$ of $\left(U_1, U_2\right)$ and the covariance matrix $\displaystyle V_{\theta, \sigma }(\widetilde{U}_1, \widetilde{U}_2)$  of any unbiased estimator $(\widetilde{U}_1, \widetilde{U}_2)$ of
$(E_{\theta, \sigma }\widetilde{U}_1, E_{\theta, \sigma }\widetilde{U}_2)$ satisfy the inequality
\begin{equation}
\label{eq:ex}
V_{\theta, \sigma }\left(U_1, U_2\right)\leq V_{\theta, \sigma }(\widetilde{U}_1, \widetilde{U}_2),\,\,\,\theta \in \mathbb{R},\,\,\,\sigma \in \mathbb{R}_{+}
\end{equation}
in the standard sense ($A \leq B$ $\Leftrightarrow$ $A-B$ is a positive semi-definite matrix.)
Note that \eqref{eq:ex} is stronger that $U_1$ and $U_2$ are separately UMVUEs.
Plainly \eqref{eq:ex} implies that a linear combination $c_1U_1+c_2U_2$ with any constants $c_1,c_2$ is a UMVUE and this is independent of the vector
$\displaystyle W=\left(\frac{X_1-\overline{X}}{S},\ldots,\frac{X_n-\overline{X}}{S}\right)$ of the standardized residuals. Independency of any linear combination $c_1U_1+c_2U_2$ of $W$ is equivalent to independence of $(U_1, U_2)$ and $W$ (this is stronger than independence of $U_1$ and $W$, $U_2$ and $W$).

If the $n$-dimensional vector $\displaystyle \left(U_1, U_2, \frac{X_1-\overline{X}}{S},\ldots,\frac{X_n-\overline{X}}{S}\right)$ is in $1-1$ correspondence with $\left(X_1,\ldots,X_n\right)$, then $(U_1,U_2)$ is a complete sufficient statistic for $(\theta, \sigma)$. Thus, for sampling from a location-scale family any estimable pair
$\left(g_{1}(\theta, \sigma), g_{2}(\theta, \sigma)\right)$ of parametric functions admits a UMVUE.

\label{se:ex}
\section*{Acknowledgements}
The authors would like to thank C.~R. Rao and Bimal~K. Roy for providing a copy of the dissertation of K. Unni defended in the Indian Statistical Institute back in 1978,
and Larry Shepp for many helpful discussions.
Pavel Chigansky made many very valuable comments on the different versions of the paper. The authors and the paper owe much to him.
Thanks to Iosif Pinelis for helpful comments and Thomas Seidman for a reference.
The Editor and an Associate Editor made many detailed comments that significantly improved the original version of the paper.
In particular, Section \ref{se:ex} owes its appearance to them.
The work of the second author was partially supported by a 2012 UMBC Summer Faculty Fellowship grant.

\appendix
\section{Proof}\label{a:aaa}
{\bf Proof of Lemma \ref{lem:KB}}
\begin{proof}
For any $C \in \mathcal{C}$, $B \in \mathcal{B}$
\begin{align*}
&
P_{\theta}(C\cap B|\mathcal{C})=E_{\theta}\left(\mathbbm{1}_{C\cap B}|\mathcal{C}\right)=\mathbbm{1}_{C}P(B),
\end{align*}
due to ancillarity of $\mathcal{B}$ and $\mathcal{P}$-independent of $\mathcal{C}$ and $\mathcal{B}$.
Similarly, for $A=\cup_{i=1}^{n}(C_i\cap B_i)$ for pairwise disjoint $C_1\cap B_1, \ldots, C_n\cap B_n$ with
$C_i \in \mathcal{C}, B_i \in \mathcal{B}, i=1,\ldots n$
\begin{align*}
&
P_{\theta}(A|\mathcal{C})=\sum_{i=1}^{n}\mathbbm{1}_{C_{i}}P(B_i)=P(A|\mathcal{C})
\end{align*}
does not depend on $\theta$. Thus for any $A$ from the algebra ${\mathcal{A}}_{0}$ generated by
$\cup_{i=1}^{n}(C_i\cap B_i)$,
\begin{align*}
&
P_{\theta}(A\cap C)=\int_{C}P_{\theta}(A|\mathcal{C})dP_{\theta}=\int_{C}P(A|\mathcal{C})dP_{\theta}.
\end{align*}
Since, $P_{\theta}(A\cap C)$ for $A \in {\mathcal{A}}_{0}$ determines $P_{\theta}(A\cap C)$
via monotone convergence for $A \in {\mathcal{A}}$, one has for any $C \in \mathcal{C}$,
\begin{align*}
&
P_{\theta}(A\cap C)=\int_{C}P_{\theta}(A|\mathcal{C})dP_{\theta}=\int_{C}P(A|\mathcal{C})dP_{\theta},\,\,\,A\in \mathcal{A},
\end{align*}
proving sufficiency of $\mathcal{C}$ for $\mathcal{P}$.
\end{proof}

\end{document}